\documentclass[11pt,twoside]{amsart}
\usepackage{amsmath, amsthm, amsfonts, amssymb}
\usepackage[mathcal,mathscr]{eucal}
\usepackage{mathrsfs}
\usepackage{graphicx, xcolor}
\usepackage[colorlinks]{hyperref}

\DeclareMathOperator{\PSL}{PSL}
\DeclareMathOperator{\SL}{SL}
\DeclareMathOperator{\Sz}{Sz}

\newtheorem{theorem}{Theorem}[section]
\newtheorem{lemma}[theorem]{Lemma}
\newtheorem{proposition}[theorem]{Proposition}
\newtheorem{corollary}[theorem]{Corollary}

\newtheorem{example}[theorem]{Example}
\newtheorem{question}[theorem]{Question}

\newtheorem{remark}[theorem]{Remark}

\begin{document}

\title{Some results on the solubility graph of a finite group}

\author[Mina Poozesh and Yousef Zamani]{Mina Poozesh and Yousef Zamani$^{\ast }$}


\thanks{$^{\ast}$ Corresponding author}

\maketitle

\begin{abstract}
Let $G$ be a finite insoluble group with soluble radical $  R(G)$. The solubility graph $\Gamma_{\rm S}(G)$ of $G$ is a simple graph whose vertices are the elements of $G\setminus  R(G) $
and two distinct vertices $x$ and $y$ are adjacent if and only if they generate a soluble subgroup of $G$. In this paper, we investigate the several properties of the solubility graph $\Gamma_{\rm S}(G)$.
\end{abstract}
\pagestyle{myheadings}
\markboth{\rightline {\scriptsize   M. Poozesh and Y. Zamani}}
         {\leftline{\scriptsize The solubility graph of a finite group}}

\bigskip
\bigskip

\section{\bf Introduction}
Let $G$ be a finite insoluble group with soluble radical $R(G)$. 
The solubility graph of $G$ is a simple graph whose vertices are the elements of $G\setminus  R(G) $
 and two vertices $u$ and $v$ are adjacent if the subgroup $\langle u, v\rangle$ is soluble. We denote this graph with $\Gamma_{\rm S}(G)$. For any $v\in G\setminus R(G)$, we denote the degree of $v$ in 
$\Gamma_{\rm S}(G)$ by $\deg(v)$. Notice that $\deg(v)=|\mathsf{Sol}_G(v)|-|R(G)|-1$. Here, $\mathsf{Sol}_G(v)$ is the solubilizer of $v$ in $G$, which consists of the elements $g$ in $G$ such that the subgroup generated by $v$ and $g$, i.e., $\langle v, g\rangle$, is soluble. For further exploration of the arithmetic and structural properties of the solubilizer of an element in a finite group $G$, we refer the reader to the references \cite{Akbari1, Akbari, Hai, Poozesh}.
Denote by $\deg (\Gamma_{\rm S}(G))$ the vertex degree set of $\Gamma_{\rm S}(G)$. According to the results presented in \cite{BHOWAL}, the solubility graph $\Gamma_{\rm S}(G)$ exhibits certain properties. It is proven that $\Gamma_{\rm S}(G)$ is not a star graph, a tree, an $n$-partite graph for any positive integer $n\geq 2$, or a regular graph. Additionally, the girth of $\Gamma_{\rm S}(G)$, which is the length of the shortest cycle in the graph, is determined to be $3$. Furthermore, the clique number of $\Gamma_{\rm S}(G)$, which represents the size of the largest complete subgraph within the graph, is proven to be at least $4$.

In a separate study conducted by Burness et al. \cite{Burness}, it is established that the solubility graph $\Gamma_{\rm S}(G)$ is a connected graph, meaning that there exists a path between any pair of vertices in the graph. Additionally, it is shown that the diameter of $\Gamma_{\rm S}(G)$, which is the maximum distance between any pair of vertices in the graph, is at most $5$.

It is indeed noteworthy to mention that the solubility graph $\Gamma_{\rm S}(G)$ can be viewed as the complement of the non-solvable graph of the group $G$ as discussed in \cite{Hai, BHOWAL2}. The non-solvable graph focuses on the relationships among the elements of $G$ that do not generate soluble subgroups. Therefore, by taking the complement, we obtain the solubility graph, which emphasizes the soluble relationships between the elements of $G$.

Moreover, the solubility graph can be seen as an extension of the commuting and nilpotent graphs of finite groups that have been extensively studied in various research works such as \cite{s.Akbari, Bates,Cameron,  Das1}. These related graphs explore specific aspects of group properties, such as commuting elements or nilpotency relationships. The solubility graph broadens the scope by focusing on the soluble subgroups generated by pairs of elements.


In this paper, the main objective is to investigate several properties of the solubility graph $\Gamma_{\rm S}(G)$.

\section{\bf Some results on vertices degree}
Let $G$ be a finite insoluble group with the soluble radical $R(G)$. We denote by $ \delta_{s} (G)$ and $\Delta_{s} (G)$, the minimum and maximum degrees of vertices in the solubility graph $\Gamma_{\rm S}(G)$,  respectively.\\
We first prove the following lemma.
\begin{lemma}\label{G-Sol}
Let $G$ be a finite insoluble group. Then for every $x\in G\setminus R(G)$, we have
\begin{itemize}
\item[(a)]
$|G|-|\mathsf{Sol}_G(x)|\geq |x|+\varphi(|x|)$,
\item[(b)]
	$|G|-|\mathsf{Sol}_G(x)|\geq 6,$
\end{itemize}
where $\varphi$ is the Euler $\varphi$-function.
\end{lemma}
\begin{proof}
Let $|x|=l$ and $y\notin \mathsf{Sol}_G(x)$. Then for each $0\leq i \leq l-1$, since $\langle x,x^{i} y\rangle=\langle x, y\rangle$, so $x^{i}y\notin \mathsf{Sol}_G(x)$.
Now suppose that $(i,l)=1$. If for $0\leq j< l$, $x^{i}y=yx^{j}$, then $y^{-1}x^{i}y=x^{j}$ and hence $y\in N_{G}(\langle x^{i}\rangle)=N_{G}(\langle x\rangle)$, a contradiction. Assume that
	$$
Y=\{ x^{i}y~|~0\leq i< l, (i,l)=1\} \cup \{yx^{j}~|~0\leq j < l\}.
$$
Then $Y\subseteq G\setminus  \mathsf{Sol}_G(x)$ and Part $(a)$ holds.
To prove Part $(b)$, suppose that $l=2$, then $Y=\{y,xy,yx\}$. Since $\langle x,y\rangle$ is insoluble, so $|y|> 2$, hence $y^{-1}\neq y$ and $y^{-1}\notin \mathsf{Sol}_G(x)$. Therefore $xy^{-1}\neq xy$. If $xy^{-1}= yx$, $y^{x}=y^{-1}$ and $\langle x,y\rangle$ is soluble because $\langle y\rangle \triangleleft\langle x,y\rangle$, which is a contradiction. Similarly, $y^{-1}x\neq yx$ and $y^{-1}x\neq xy$. Therefore
 $$
\{y, xy, yx, y^{-1}, xy^{-1}, y^{-1}x\}\subseteq G\setminus  \mathsf{Sol}_G(x).
$$
For $l=3$, by Part $(a)$, $|G|-|\mathsf{Sol}_G(x)|\geq5$ and the equality is not possible because
$l$ divides $|\mathsf{Sol}_{G}(x)|$. Eventually by Part $(a)$, for $l\ge 4$, $|G|-|\mathsf{Sol}_G(x)|\geq 6$ and the proof is complete. 
\end{proof}
In the following we give lower and upper bounds for $\delta_{s}(G)$ and $\Delta_{s}(G)$, respectively.
\begin{proposition}\label{delta}
In the solubility graph $\Gamma_{\rm S}(G)$, if $R(G)=1$, then $ \delta_{s}(G) \geq 8$; otherwise $ \delta_{s}(G) \geq 17$. 
\end{proposition}
\begin{proof}
Let $|\mathsf{Sol}_G(x)|=s$ and $|R(G)|=r$. Then for any $x\in G\setminus R(G)$, $\deg(x)=s-r-1$. By \cite[Corollary 4.7]{Akbari1},  $s\ge 10$. So if $r=1$,  then $\deg(x)\ge 8$. Assume that $r\neq 1$. Since 
$$
|\mathsf{Sol}_{G/R(G)}(xR(G))|=\frac{|\mathsf{Sol}_G(x)|}{|R(G)|}=\frac{s}{r}\geq 10,
$$
so $\deg(x)\ge 9r-1$. Now $r\ge 2$, implies that $\deg(x)\ge 17$.
\end{proof}
\begin{remark}
By GAP \cite{GAP} we can see that if $G=A_{5}\times \mathbb{Z}_{2}$  and $x$ is an element of order $5$ in  $G$, then $\deg(x)=17$. Also for $G=A_{5}$ and $x=(1~2~3~4~5)\in A_{5}$, we have $\deg(x)=8$.
This example shows that the bounds in Proposition \ref{delta} are sharp.
\end{remark}

\begin{proposition}\label{D}
In the solubility graph $\Gamma_{\rm S}(G)$, $\Delta_{s}(G)\leq n-7$, where $n$ is the number of vertices in $\Gamma_{\rm S}(G)$.
\end{proposition}
\begin{proof}
We know that there exists $n=|G|-|R(G)|$ vertices in $\Gamma_{\rm S}(G)$. Assume that $v$ is an arbitrary vertex in $\Gamma_{\rm S}(G)$.  Applying Lemma \ref{G-Sol}, we have
$$
\deg(v)=|\mathsf{Sol}_G(v)|-|R(G)|-1=(|\mathsf{Sol}_G(v)|-|G|)+n-1\leq n-7,
$$
so the result holds.
\end{proof}
\begin{proposition}\label{deg p-1}
Let $G$ be a finite insoluble group and $\Gamma_{\rm S}(G)$ be its solubility graph. If for a prime number such as $p$, the degree of a vertex of
$\Gamma_{\rm S}(G)$ is $p-1$, then $R(G)=1$.
\end{proposition}
\begin{proof}
Assume that $x$ is a vertex of $\Gamma_{\rm S}(G)$ with degree $p-1$. Then 
$|\mathsf{Sol}_G(x)|-|R(G)|=p$. Since $|R(G)|$ divides $|\mathsf{Sol}_G(x)|$, so $|R(G)|=1~ \text{or}~ p$.  If $|R(G)|=p$, 
then $|\mathsf{Sol}_G(x)|=2p$ and $|\mathsf{Sol}_{G/R}(xR)|=\frac{|\mathsf{Sol}_G(x)|}{|R(G)|}=2,$ a contradiction. So $|R(G)|=1$ and the result holds.
\end{proof}

\begin{proposition}
Let $G$ be a finite group such that for all $x\in G\setminus R(G)$, $|\mathsf{Sol}_G(x)|\ge \frac{|G|+|R(G)|}{2}+1$. Then $\Gamma_{\rm S}(G)$ is Hamiltonian.	
\end{proposition}
\begin{proof}
Let $v\in G\setminus R(G)$. Then
\begin{align}
\deg(v)&=|\mathsf{Sol}_G(v)|-|R(G)|-1\nonumber\\
&\ge \frac{|G|+|R(G)|}{2}-|R(G)|\nonumber \\
&=\frac{|G|-|R(G)|}{2},\nonumber
\end{align}
so the result holds by Dirac's theorem (see \cite{Bondy}).
\end{proof}
\begin{proposition}
Let $G$ be a finite insoluble group. Then for every $v\in G\setminus R(G)$, we have
$$
\frac{1+\deg (v)}{1+\deg (vR(G))}=|R(G)|.
$$

Furthermore $\deg (\Gamma_{\rm S}(G))=\deg (\Gamma_{\rm S}(G/R(G)))$.
\end{proposition}
\begin{proof}
We have
\begin{eqnarray*}
1+\deg (v)&=&|\mathsf{Sol}_G(v)|-|R(G)|\\
&=&|R(G)||\mathsf{Sol}_{G/R(G)}(vR(G))|-|R(G)|\\
&=&|R(G)|(|\mathsf{Sol}_{G/R(G)}(vR(G))|-1)
\end{eqnarray*}
Also we have
$$
1+\deg (vR(G))=|\mathsf{Sol}_{G/R(G)}(vR(G))|-1.
$$
By combining the two previous relations, the result is obtained.
\end{proof}
\section{\bf The number of edges and the solubility degree}
The solubility degree of a finite group $G$ (see \cite{BHOWAL}) is the probability that two randomly chosen elements of $G$ generate a soluble group. It is  is  given by  
$P_{s}(G)=\frac{|\mathbb{S}|}{|G|^2}$, where 
$$
\mathbb{S}=\{ (x, y)\in G\times G ~|~ \langle x, y\rangle ~ \text{is soluble}\}.
$$
 It is not difficult to see that $|\mathbb{S}|=\sum_{x\in G}|\mathsf{Sol}_{G}(x)|$. Thus 
$$
P_{s}(G)=\frac{1}{|G|^2}\sum_{x\in G}|\mathsf{Sol}_{G}(x)|.
$$
  Notice that $G$ is soluble if and only if $P_{s}(G)=1$. It is known (see \cite{Guralnick2}) that $P_{s}(G)\leq \frac{11}{30}$. By using GAP \cite{GAP}, we can see that $ P_{s}(A_5)=\frac{11}{30}$, which shows that the bound is sharp. In \cite{BHOWAL}, it is proved that if $G$ is a finite group, then $P_{s}(G)\geq Pr(G)$ and the equality holds if and only if $G$ is a soluble group, where $Pr(G)$ is the commutativity degree of $G$. However, we provide a counterexample where the equality condition is not true. Additionally, we provide a condition for equality in the theorem.
\begin{example}
Consider $G=S_{3}$. Then $Pr(G)=\frac{k(G)}{|G|}=\frac{1}{3}$, where $k(G)$ is the number of the conjugacy classes of $G$. Since $G$ is soluble, we have 
$P_{s}(G)=1$.
\end{example}

\begin{proposition}\label{solubility-commutativity}
Let $G$ be a finite group. Then $P_{s}(G)\geq Pr(G)$ and the equality holds if and only if $G$ is an abelian group, where $Pr(G)$ is the commutativity degree of $G$.
\end{proposition}
\begin{proof}
We know that $Pr(G)=\frac{1}{|G|^2}\sum_{x\in G}|\mathcal{C}_{G}(x)|$. Now $\mathcal{C}_{G}(x)\subseteq \mathsf{Sol}_{G}(x)$ implies that $P_{s}(G)\geq Pr(G)$. If 
$P_{s}(G)=Pr(G)$, then $|\mathcal{C}_{G}(x)|=|\mathsf{Sol}_{G}(x)|$, for every $x\in G$. This implies that for every $x\in G$, $\mathsf{Sol}_{G}(x)$ is a subgroup of $G$. Thus $G$ is soluble by \cite[Proposition 2.22]{Hai}. Therefore
$$
G=R(G)=\bigcap_{x\in G}\mathsf{Sol}_{G}(x)=\bigcap_{x\in G} \mathcal{C}_{G}(x)=Z(G),
$$
 and the result holds. The converse is obvious.
\end{proof}
\begin{proposition}\label{quotiont degree}
Let $G$ be an insoluble finite group and $N$ a normal subgroup of $G$. Then
$P_{s}(G)\leq P_{s}(G/N)$. Furthermore, if $N$ is soluble, then the equality holds.
\end{proposition}
\begin{proof}
By \cite[Lemma 2.3]{Poozesh}, we have
	\begin{align}
	|G/N|^2P_{s}(G/N)&=\sum_{xN\in G/N}|\mathsf{Sol}_{G/N}(xN)|\nonumber\\
&=\frac{1}{|N|}\sum_{x\in G}|\mathsf{Sol}_{G/N}(xN)|\nonumber\\
&\geq  \frac{1}{|N|}\sum_{x\in G}\frac{|\mathsf{Sol}_{G}(x)|}{|N|}\nonumber
\\
&=\frac{|G|^2}{|N|^2} P_{s}(G),\nonumber
\end{align}
so the result holds. If $N$ is soluble, then by \cite[Lemma 2.3]{Poozesh} we have $|\mathsf{Sol}_{G/N}(xN)|=\frac{|\mathsf{Sol}_{G}(x)|}{|N|}$, so the equality holds.
\end{proof}
Applying Proposition \ref{quotiont degree}, we deduce the following corollary.
\begin{corollary}
If $G/R(G)\cong H/R(H)$, then $P_{s}(G)=P_{s}(H)$.
\end{corollary}
\begin{proposition}
If $G$ and $H$ are two finite groups, then $P_{s}(G\times H)\geq P_s(G) P_s(H)$. If $G$ or $H$ is soluble (in particular $(|G|,|H|)=1$), then the equality holds.
\end{proposition}
\begin{proof}
For any $(g, s), (x, t)\in G\times H$, we have 
$$
\langle (g, s), (x, t)\rangle\subseteq \langle g, x\rangle \times \langle t, s\rangle.
$$
 So 
$$
\mathsf{Sol}_{G}(g)\times \mathsf{Sol}_{H}(s)\subseteq \mathsf{Sol}_{G\times H}(g, s)
$$
 Thus
\begin{eqnarray*}
P_s(G) P_s(H)&=&\frac{1}{|G|^2}\sum_{x\in G}|\mathsf{Sol}_{G}(x)|\frac{1}{|H|^2}\sum_{t\in H}|\mathsf{Sol}_{H}(t)|\\
&=&\frac{1}{|G|^2 |H|^2}\sum_{x\in G}\sum_{t\in H}|\mathsf{Sol}_{G}(x)||\mathsf{Sol}_{H}(t)|\\
&\leq& \frac{1}{|G\times H|^2}\sum_{(x,t)\in G\times H}|\mathsf{Sol}_{G\times H}(x, t)|\\
&=&P_s(G\times H).
\end{eqnarray*}
If $H$ is soluble, then by Proposition \ref{quotiont degree}, we have
$$
P_s(G\times H)\leq P_s(G\times H/H)=P_s(G)=P_s(G) P_s(H).
$$
Therefore we obtain the equality.
\end{proof}
Let $|E(\Gamma_{\rm S}(G))|$ be the number of edges of the graph $\Gamma_{\rm S}(G)$. The next proposition gives a relation between $|E(\Gamma_{\rm S}(G))|$ and $P_{s}(G)$. Also. it gives a sufficient condition for the solubility graph $\Gamma_{\rm S}(G)$ to be Hamiltonian.
\begin{proposition}\label{No Edges}
Let $G$ be a finite insoluble group. Then
\begin{itemize}
\item[{\bf (i)}] $2|E(\Gamma_{\rm S}(G))|=|G|^2P_s(G)+|R(G)|^2+|R(G)|-|G|(2|R(G)|+1)$.
\item[{\bf (ii)}] If $P_s(G)\ge 1-\frac{2}{|G|}+\frac{2|R(G)|}{|G|^2}+\frac{4}{|G|^2}$, then $\Gamma_{\rm S}(G)$ is Hamiltonian.
\end{itemize}

\end{proposition}
\begin{proof}
{\bf(i)} 
\begin{align}
2|E(\Gamma_{\rm S}(G))|&=\sum_{v\in G\setminus R(G)}\deg (v)\nonumber\\
&=\sum_{v\in G\setminus R(G)}(|\mathsf{Sol}_{G}(v)|-|R(G)|-1)\nonumber\\
&=\sum_{v\in G}|\mathsf{Sol}_{G}(v)|-|G||R(G)|-(|G|-|R(G)|)(|R(G)|+1)\nonumber\\
&=|G|^2P_s(G)+|R(G)|^2+|R(G)|-|G|(2|R(G)|+1)\nonumber
\end{align}
{\bf(ii)} By part $(i)$, we have
\begin{align}
2|E(\Gamma_{\rm S}(G))|&=|G|^2P_s(G)+|R(G)|^2+|R(G)|-|G|(2|R(G)|+1)\nonumber\\
&\geq |G|^2(1-\frac{2}{|G|}+\frac{2|R(G)|}{|G|^2}+\frac{4}{|G|^2})+|R(G)|^2+|R(G)|-2|R(G)||G|-|G| \nonumber\\
&=(|G|-|R(G)|)^2-3(|G|-|R(G)|)+4\nonumber\\
&=2\dbinom{|G|-|R(G)|-1}{2}+2,\nonumber
\end{align}
therefore 
$$
|E(\Gamma_{\rm S}(G))|\ge \dbinom{|G|-|R(G)|-1}{2}+1.
$$
 Then Ore-Bondy Corollary \cite{Bondy} implies that the solubility graph is Hamiltonian.
\end{proof}
\begin{proposition}\label{|E|>}
	Let $G$ be an insoluble finite group with $R(G)=1$.  Then 
	$$|E(\Gamma_{\rm S}(G))|\geq \frac{|G|}{2}(k(G)-3)+1,$$
and the equality holds if and only if $G$ is abelian.
	\end{proposition}
	\begin{proof}
By using Proposition \ref{No Edges} and Proposition \ref{solubility-commutativity}, we have
\begin{align}
	2|E(\Gamma_{\rm S}(G))|&=|G|^2P_s(G)+|R(G)|^2+|R(G)|-|G|(2|R(G)|+1)\nonumber\\
	&\geq |G|^{2}Pr(G)+2-3|G|\nonumber\\
		&=  |G|^2\frac{k(G)}{|G|}+2-3|G|\nonumber\\
		&=|G|(k(G)-3)+2.\nonumber
\end{align}
The equality holds if and only if $P_s(G) = Pr(G)$, which is equivalent to the condition that $G$ is an abelian group, as stated in Proposition \ref{solubility-commutativity}.
\end{proof}

In the following we give some lower and upper bounds for the number of edges of the graph $\Gamma_{\rm S}(G)$.

\begin{proposition}\label{|E|>|G|+1}
Let $G$ be an insoluble finite group with $R(G)=1$. Then $|E(\Gamma_{\rm S}(G))|> |G|+1$.  
\end{proposition}
\begin{proof}
Since $G$ is insoluble, we have $k(G) \geq 5$. Thus, $|E(\Gamma_{\rm S}(G))| \geq |G| + 1$ by Proposition \ref{|E|>}. If $|E(\Gamma_{\rm S}(G))| = |G| + 1$, then again by Proposition \ref{|E|>}, we would have $k(G) = 5$. According to \cite[Note A]{Burnside}, this would imply $G \cong A_5$, which is impossible because $|E(\Gamma_{\rm S}(A_5))| = 571$ by using Proposition \ref{No Edges}.
\end{proof}

\begin{proposition}
Let $G$ be a finite insoluble simple group. Then $|E(\Gamma_{\rm S}(G))|>4|G|+1$.
\end{proposition}
\begin{proof}
Assume that $|E(\Gamma_{\rm S}(G))| \leq 4|G| + 1$. Then, by Proposition \ref{|E|>}, we have $k(G) \leq 11$. The groups with such a property have been classified (see \cite[Tables 1-3]{Lopez}). Since $G$ is simple, it follows that $G$ is isomorphic to one of the following groups:

$$
A_{6},A_{7},\PSL(2, q)~(q=7,11,13,17), \PSL(3,4),M_{11},\Sz(8).
$$
Using GAP \cite{GAP}, we can check that for any of the above groups, $|E(\Gamma_{\rm S}(G))| > 4|G| + 1$, which leads to a contradiction.
\end{proof}
Applying the upper bound $\frac{11}{30}$ for the solubility degree and Proposition \ref{No Edges}, we immediately obtain the following corollary.\begin{corollary}\label{|E|<}
Let $G$ be a finite insoluble group with $R(G)=1$. Then   
$$|E(\Gamma_{\rm S}(G))|\leq \frac{11}{60}|G|^2-\frac{3}{2}|G|+1.$$
\end{corollary}
By Proposition \ref{No Edges}, we see that if $G=A_{5}$, then $|E(\Gamma_{\rm S}(G))|=571$. Therefore the upper bound in the previous corollary is sharp.

\section{\bf Groups with isomorphic solubility graphs}
Let $G$ and $H$ be two finite insoluble groups. A graph isomorphism between the solubility graphs $\Gamma_{\rm S}(G)$ and $\Gamma_{\rm S}(H)$ is a one-to-one correspondence $\phi: G \setminus R(G) \rightarrow H \setminus R(H)$ such that $\phi$ preserves edges. In other words, if $x, y \in G \setminus R(G)$ and $\langle x, y \rangle$ is soluble, then $\langle \phi(x), \phi(y) \rangle$ is soluble as well. We denote the isomorphism by $\Gamma_{\rm S}(G) \cong \Gamma_{\rm S}(H)$.

It is worth noting that if $G \cong H$, then $\Gamma_{\rm S}(G) \cong \Gamma_{\rm S}(H)$. However, there are cases where groups are not isomorphic, yet their solubility graphs are isomorphic. For example, we have $\SL(2,5) \not\cong \mathbb{Z}2 \times A_5$, but by applying GAP \cite{GAP},
 we can observe that $\Gamma_{\rm S}(\SL(2,5)) \cong \Gamma_{\rm S}(\mathbb{Z}_2 \times A_5)$.
\begin{proposition}\label{iso1}
There is no finite insoluble group $G$ with an insoluble proper subgroup $H$ such that $\Gamma_{\rm S}(G)\cong\Gamma_{\rm S}(H)$.
\end{proposition}
\begin{proof}
Contrarily, let us assume that there exists a finite insoluble group $G$ with an insoluble proper subgroup $H$ such that $\Gamma_{\rm S}(G) \cong \Gamma_{\rm S}(H)$. This implies that the vertex sets of $\Gamma_{\rm S}(H)$ and $\Gamma_{\rm S}(G)$ coincide. Therefore, we have
\begin{eqnarray}
|G| - |R(G)| = |H| - |R(H)|.
\end{eqnarray}
Since $G$ is insoluble, $R(G)$ is a proper subgroup of $G$, which implies that $[G:R(G)] \geq 2$. Consequently, we have $|R(G)| \leq \frac{|G|}{2}$. Moreover, since $H$ is a proper subgroup of $G$, we have $[G:H] \geq 2$. Now, Equation (4.1) implies
$$
\frac{|G|}{2}\leq|H|-|R(H)|\leq \frac{|G|}{2}-|R(H)|,
$$
 which is impossible. 
\end{proof}
\begin{proposition}
There is no finite insoluble group $G$ with a non-trivial normal subgroup $N$ such that the quotient group $G/N$ is insoluble and $\Gamma_{\rm S}(G) \cong \Gamma_{\rm S}(G/N)$.
\end{proposition}
\begin{proof}
Contrarily, let's assume that there exists a finite insoluble group $G$ with a non-trivial normal subgroup $N$ such that the quotient group $G/N$ is insoluble and $\Gamma_{\rm S}(G) \cong \Gamma_{\rm S}(G/N)$. This implies that the vertex sets of $\Gamma_{\rm S}(G/N)$ and $\Gamma_{\rm S}(G)$ correspond to each other. Therefore, we have
\begin{eqnarray}
|G| - |R(G)| = |G/N| - |R(G/N)|.
\end{eqnarray}
Since $N$ is a non-trivial subgroup of $G$, we have $|G/N| \leq \frac{|G|}{2}$. Similar to the previous proposition, we also have $|R(G)| \leq \frac{|G|}{2}$. Now, using Equation (4.2), we obtain
$$
\frac{|G|}{2}\leq |G/N|-|R(G/N)|\leq \frac{|G|}{2}-|R(G/N)|,
$$
which leads to a contradiction. Thus, the proof is complete.
\end{proof}
\begin{proposition}\label{|G|=|H|}
If $G$ and $H$ are two finite insoluble simple groups such that $\Gamma_{\rm S}(G) \cong \Gamma_{\rm S}(H)$ then $|G|=|H|$.
\end{proposition}
\begin{proof}
Since $\Gamma_{\rm S}(G) \cong \Gamma_{\rm S}(H)$, so the vertex set of $\Gamma_{\rm S}(G)$  corresponds to the vertex set of $\Gamma_{\rm S}(H)$. Therefore, we have $|G|-|R(G)|=|H|-|R(H)|$. Consequently, $|G|=|H|$.
\end{proof}
\begin{theorem}
Let $G$ be finite insoluble group. Assume that $n\geq 5$ is a natural number such that $p=\frac{n!}{2}-1$ is a prime number. If $\Gamma_{\rm S}(G) \cong \Gamma_{\rm S}(A_n)$ then $|G|=|A_n|$.
\end{theorem}
\begin{proof}
Suppose $\Gamma_{\rm S}(G) \cong \Gamma_{\rm S}(A_n)$. Then the vertex set of $\Gamma_{\rm S}(G)$ corresponds to the vertex set of $\Gamma_{\rm S}(A_n)$. Therefore, we have $|G|-|R(G)|=p$. However, since $|R(G)|$ divides $|G|$, we find that $|R(G)|~\mid ~p$, implying $|R(G)|=1$ or $p$. If $|R(G)|=p$, then $|G|=2p$, which is impossible. Hence, we conclude that $|R(G)|=1$, and thus $|G|=|A_n|$.
\end{proof}
\begin{corollary}\label{G, A_{5}}
Let $G$ be a finite insoluble group such that  $\Gamma_{\rm S}(G) \cong \Gamma_{\rm S}(A_5)$. Then $G\cong A_{5}$.
\end{corollary}
\begin{corollary}
Let $G$ be a finite insoluble group such that  $\Gamma_{\rm S}(G) \cong \Gamma_{\rm S}(A_6)$. Then $|G|=|A_6|$ and $G$ is isomorphic with one of the following groups:
$$
A_6,~C_{3}\rtimes S_{5},~C_{3}\times S_{5},~S_{3}\times A_{5},~C_{6}\times A_{5},~\text{or}~C_{3}\times \SL(2, 5).
$$
Furthermore, if $G$ is quasisimple or almost simple group, then $G\cong A_{6}$.
\end{corollary}
\begin{proposition}
	Let $G$ be a  finite insoluble group such that $\Gamma_{\rm S}(G) \cong \Gamma_{\rm S}(A_7)$. Then
	$$|G|=|A_{7}|$$
\end{proposition}
\begin{proof}
Suppose $\Gamma_{\rm S}(G) \cong \Gamma_{\rm S}(A_7)$. Then the vertex set of $\Gamma_{\rm S}(G)$ corresponds to the vertex set of $\Gamma_{\rm S}(A_7)$. Therefore, we have $|G|-|R(G)|=2519$. However, note that $|R(G)|$ must divide $|G|$. Thus, we have $|R(G)|\mid 2519$. This implies that $|R(G)|$ can only be equal to $1$, $11$, $229$, or $11 \cdot 229$. Therefore, we have four possible values for $\frac{|G|}{|R(G)|}$: $2520$, $23\cdot 5\cdot 2$, $3\cdot 2^2$, or $2$. Given that $G$ is insoluble, we can conclude that
 $|G|=|A_7|=2520$. Therefore, the result holds.
\end{proof}

\begin{remark}
Similar to the proof of the previous theorem and using GAP, it has been observed that if $G$ is a finite insoluble group such that $\Gamma_{\rm S}(G) \cong \Gamma_{\rm S}(A_n)$ ($n\leq 12$), then $|G|=|A_n|$ \end{remark}
Now the following question arises.
\begin{question}
Let $G$ be a finite insoluble group such that $\Gamma_{\rm S}(G) \cong \Gamma_{\rm S}(A_n)$. Is it true that $|G|=|A_n|$?
\end{question}
We can pose the following question in general:
\begin{question}
Let $H$ be a finite insoluble simple group. Let $G$ be a finite insoluble group such that $\Gamma_{\rm S}(G) \cong \Gamma_{\rm S}(H)$. Is it true that $|G|=|H|$?
\end{question}
 By using Artin Theorem (\cite{Artin 1, Artin 2}) and the classification of finite simple groups , we have the following theorem.
\begin{theorem}\cite[Lemma 2.3]{Derafsheh 1}\label{Artin}
Let $G$ and $H$ be finite simple groups, $|G|=|H|$, then the following holds:
\begin{itemize}
\item[(a)]
If $|G|=|A_8|=|\PSL(3,4)|$, then $G\cong A_8$ or $G\cong \PSL(3,4)$;
\item[(b)]
If $|G|=|B_n(q)|=|C_n(q)|$, where $n\geq 3$ and $q$ is odd, then $G\cong B_n(q)$ or $G\cong C_n(q)$;
\item[(c)]
If $H$ is not in the above cases, then $G\cong H$.
\end{itemize}
\end{theorem}

\begin{corollary}\label{cong}
Let $G$ and $H$ be finite simple groups such that $\Gamma_{\rm S}(G) \cong \Gamma_{\rm S}(H)$. The following statements hold:
\begin{itemize}
\item[(a)]
If $H\in \{ A_8, \PSL(3,4) \}$, then $G\cong A_8$ or $G\cong \PSL(3,4)$.
\item[(b)]
If $H\in \{B_n(q), C_n(q) \}$, where $n\geq 3$ and $q$ is odd, then $G\cong B_n(q)$ or $G\cong C_n(q)$.
\item[(c)]
If $H$ does not fall into the above cases, then $G\cong H$.
\end{itemize}
\end{corollary}
\begin{proof}
The result immediately follows from Proposition \ref{|G|=|H|} and Theorem \ref{Artin}.
\end{proof}

As an immediate consequence of Corollary \ref{cong}, we get the following corollary.
\begin{corollary}
Let $G$ be a finite insoluble simple group such that  $\Gamma_{\rm S}(G) \cong \Gamma_{\rm S}(A_n)$, where $n$ is a natural number, $n\geq 5$, $n\neq 8$. 
Then $G\cong A_n$.
\end{corollary}
\section{\bf Degree pattern}
Now we  introduce the degree pattern of the solubility graph
and use it to give a conjecture. 
{\em The degree pattern} of the solubility graph of $G$, denoted as ${\rm D}_{s}(G)$, is represented by an $n$-tuple:
$$
{\rm D}_{s}(G)=(d_1, d_2, \cdots,
d_n),
$$
 where $d_1 \geq d_2 \geq \cdots \geq d_n$, and these values correspond to the degrees of vertices in the solubility graph $\Gamma_{\rm S}(G)$. Here, $n$ is determined by $n = |G| - |R(G)|$.
\begin{proposition}\label{degree pattern}
Let $G$ be a finite insoluble group, and let ${\rm D}{s}(G)=(d_1, d_2, \cdots,d_n)$ denote the degree pattern of the solubility graph of $G$.
Then, there exist indices $i$ and $j$ such that $d_{i}\neq d_{j}$.\end{proposition}
\begin{proof}
First, we prove the assertion for the case $|R(G)|=1$. Assume that for all $1\leq i\leq n$,~$d_{i}=d$. Then
	 \begin{eqnarray*}
	 	d(|G|-1)&=&\sum_{i=1}^{|G|-1} d_i\nonumber\\
	 	&=&\sum_{x\in G\setminus \{1\}} (|\mathsf{Sol}_G(x)|-2)\nonumber\\
	 	&=& \sum_{x\in G\setminus \{1\}} |\mathsf{Sol}_G(x)|-2|G|+2\nonumber\\
	 	&=&\sum_{x\in G} |\mathsf{Sol}_G(x)|-3|G|+2.\nonumber
	 \end{eqnarray*}
Since $|G|~\Big{|}~ \sum_{x\in G} |\mathsf{Sol}_G(x)|$, we have 
$|G|~\big{|}~ 2+d $.This implies that $|G|$ divides $|\mathsf{Sol}_G(x)|$ for each $x\in G$. As a result, we have $G = \mathsf{Sol}_G(x)$ for all $x\in G$. However, this contradicts the assumption that $G$ is an insoluble group. Therefore, the assumption that all $d_i$ are equal leads to a contradiction.\\

Now, let $R=R(G)\neq\{1\}$ and ${\rm D}_{s}(G/R)=(d^{\prime}_1, d^{\prime}_2, \cdots,d^{\prime}_m)$ be the degree pattern of the solubility graph of $G/R$.
 Suppose $\deg_{G/R}(x_{i}R)=d^{\prime}_i,~1\leq i\leq m$. Since $R(G/R)=1$, so there exist the indices 
$1\leq i< j\leq m$ such that $d^{\prime}_{i}\neq d^{\prime}_{j}$. Then, we have  
$$
|\mathsf{Sol}_{G/R}(x_{i}R)|-2 \neq |\mathsf{Sol}_{G/R}(x_{j}R)|-2,
$$ 
which implies $\frac{|\mathsf{Sol}_{G/R}(x_{i}R)|}{|R|}\neq \frac{|\mathsf{Sol}_{G/R}(x_{j}R)|}{|R|}$.  Consequently, we obtain
 $|\mathsf{Sol}_G(x_{i})|\neq |\mathsf{Sol}_G(x_{j})|$, and thus $\deg_{G}(x_i)\neq\deg_{G}(x_j)$. This implies that there exist the indices $1\leq k<l\leq n$ such that $d_{k}=\deg_{G}(x_i)\neq\deg_{G}(x_j)=d_{l}$, and the result holds.
\end{proof}
As an immediate consequence of Proposition \ref{degree pattern} we have the following corollary.
\begin{corollary}
Let $G$ be a finite insoluble group. Then $\Gamma_{\rm S}(G)$ is not regular.	
\end{corollary}
We will now determine the structure of a finite group $G$ based on the degrees of vertices in its solubility graph $\Gamma_{\rm S}(G)$. We pose the following question:

\begin{question}
 Let $G$ and $H$ be two insoluble groups such that $\rm{D}{s}(G) = \rm{D}{s}(H)$. Is it true that $|G| = |H|$?
\vspace{.03cm}
\end{question}
In the following proposition, we demonstrate that the question holds true when one of the degrees of vertices in $\rm{D}_{s}(G) = \rm{D}_{s}(H)$ is a specific number.
\begin{proposition}\label{D_s(G)=D_s(H)}
Let $G$ and $H$ be two insoluble groups such that $\rm{D}_{s}(G) = \rm{D}_{s}(H)$. Furthermore, let one of the degrees of vertices be $p-1$, where $p$ is a prime, and $n = |G| - |R(G)|$. Then $|G| = |H|$.
\end{proposition}
\begin{proof}
By assumption there are two elements $x\in G$ and $y\in H$ such
	that ${\rm deg}(x)={\rm deg}(y)=p-1$. Then using Theorem \ref{deg p-1} we deduce that $R(G)=R(H)=1$. The rest of proof is obvious.
\end{proof}
Finally, we utilize Proposition \ref{D_s(G)=D_s(H)} and Theorem \ref{Artin} to deduce the following corollary:

\begin{corollary}
Let $G$ be a finite insoluble simple group, excluding $A_{8}$, $\PSL(3,4)$, $B_{n}(q)$, and $C_n(q)$, where $q$ is odd and $n\geq 3$. Additionally, let $H$ be an insoluble group such that $\rm {D}_{s}(G) = \rm{D}_{s}(H)$. If one of the degrees of vertices is $p-1$, where $p$ is a prime, then $H\cong G$.
\end{corollary}


{\footnotesize}{\bf M. Poozesh}, {Department of Mathematics}, {Faculty of Basic Sciences}, {Sahand University of Technology}, {Tabriz, Iran}\\
{\tt Email: mi\_poozesh@sut.ac.ir}\\

{\footnotesize}{\bf Y. Zamani}, {Department of Mathematics}, {Faculty of Basic Sciences}, {Sahand University of Technology}, {Tabriz, Iran}\\
{\tt Email: zamani@sut.ac.ir}
\end{document}